\documentclass{amsart}
\usepackage[active]{srcltx}
\usepackage{amsthm,mathrsfs}
\usepackage{a4wide}
\usepackage[english]{babel}
\usepackage{amsfonts}
\usepackage{amsmath}
\usepackage{amssymb}
\usepackage{dsfont}
\newtheorem{lemma}{Lemma}
\newtheorem{theorem}{Theorem}
\newtheorem{corollary}{Corollary}

\usepackage{float}
\usepackage{url}

\newcommand {\E} {\mathcal{E}}

\newcommand {\p} {\mathbb{P}}
\newcommand {\Z} {\mathbb{Z}}
\newcommand {\N} {\mathbb{N}}

\newcommand {\R} {\mathbb{R}}

\makeatletter\def\blfootnote{\xdef\@thefnmark{}\@footnotetext}\makeatother
\allowdisplaybreaks
\title{\bf A note on the Duffin-Schaeffer conjecture with slow divergence}

\author{Christoph Aistleitner} 
\address{Department of Applied Mathematics, School of Mathematics and Statistics, University of New South Wales, Sydney NSW 2052, Australia}
\email{aistleitner@math.tugraz.at}

\thanks{The author is supported by a Schr\"odinger scholarship of the Austrian Research
Foundation (FWF)}


\begin{document}

\begin{abstract}
For a non-negative function $\psi: ~ \N \to \R$, let $W(\psi)$ denote the set of real numbers $x$ for which the inequality $|n x - a| < \psi(n)$  has infinitely many coprime solutions $(a,n)$. The Duffin--Schaeffer conjecture, one of the most important unsolved problems in metric number theory, asserts that $W(\psi)$ has full measure provided
\begin{equation} \label{dsccond}
\sum_{n=1}^\infty \frac{\psi(n) \varphi(n)}{n} = \infty.
\end{equation}
Recently Beresnevich, Harman, Haynes and Velani proved that $W(\psi)$ has full measure under the \emph{extra divergence} condition
$$
\sum_{n=1}^\infty \frac{\psi(n) \varphi(n)}{n \exp(c (\log \log n) (\log \log \log n))} = \infty \qquad \textrm{for some $c>0$}.
$$
In the present note we establish a \emph{slow divergence} counterpart of their result: $W(\psi)$ has full measure, provided~\eqref{dsccond} holds and additionally there exists some $c>0$ such that
$$
\sum_{n=2^{2^h}+1}^{2^{2^{h+1}}} \frac{\psi(n) \varphi(n)}{n} \leq \frac{c}{h} \qquad \textrm{for all \quad $h \geq 1$.}
$$
\end{abstract}

\date{}
\maketitle

\section{Introduction}

For a non-negative function $\psi: ~ \N \to \R$ we define sets $\E_n \subset \R / \Z$ by
\begin{equation} \label{en}
\E_n = \bigcup_{\substack{1 \leq a \leq n,\\\gcd(a,n)=1}} \left( \frac{a-\psi(n)}{n},\frac{a+\psi(n)}{n} \right)
\end{equation}
and write $W(\psi)$ for the limsup-set
\begin{equation} \label{wn}
W(\psi) = \limsup_{N \to \infty} \E_N := \bigcap_{N=1}^\infty \bigcup_{n =N}^\infty \E_n.
\end{equation}
An important open problem in metric Diophantine approximation is to specify under which conditions on $\psi$ we have $\lambda(W(\psi))=1$, that is, under which conditions imposed on $\psi$ for almost all $x \in \R / \Z$ (in the sense of Lebesgue measure) the inequality
$$
|n x - a| < \psi(n)  
$$
has infinitely many coprime solutions $(a,n)$. It is easy to see that the condition
\begin{equation} \label{dsc}
\sum_{n=1}^\infty \frac{\psi(n) \varphi(n)}{n} = \infty
\end{equation}
is \emph{necessary} to have $\lambda(W(\psi))=1$; a famous conjecture, stated by Duffin and Schaeffer~\cite{ds} in 1941, claims that this condition is also \emph{sufficient}. Recently Beresnevich, Harman, Haynes and Velani~\cite{bhhv} proved that $\lambda(W(\psi))=1$ under the \emph{extra divergence} condition
\begin{equation} \label{in}
\sum_{n=1}^\infty \frac{\psi(n) \varphi(n)}{n \exp(c (\log \log n)(\log \log \log n))} = \infty \qquad \textrm{for some $c>0$}.
\end{equation}
The purpose of the present note is to prove a \emph{slow divergence} version of the Duffin--Schaeffer conjecture. For $h \geq 1$ we set
\begin{equation}\label{deltah}
\Delta_h = \left\{ 2^{2^h}+1, \dots, 2^{2^{h+1}} \right\} \qquad \textrm{and} \qquad S_h = \sum_{n \in \Delta_h} \frac{\psi(n) \varphi(n)}{n}.
\end{equation}

\begin{theorem} \label{th1}
We have $\lambda(W(\psi))=1$, provided~\eqref{dsc} holds and there exists a constant $c$ such that
\begin{equation} \label{condh}
S_h \leq \frac{c}{h} \qquad \textrm{for all \quad $h \geq 1$.}
\end{equation}
\end{theorem}

The partitioning of the index set $\N$ into blocks $\Delta_h$ of the form~\eqref{deltah} is quite natural, since it implies that for indices $m$ and $n$ from non-adjacent blocks the sets $\E_m$ and $\E_n$ are quasi-independent (see for example~\cite{bhhv,hpv} and the proof of Theorem~\ref{th1} below). In condition~\eqref{dsc} we can assume without loss of generality that $S_h \leq 1$, so it would be very interesting to see if~\eqref{condh} can be further relaxed.\\

As a consequence of Theorem \ref{th1} we immediately get the following corollary.

\begin{corollary} \label{co1}
We have $\lambda(W(\psi))=1$, provided
\begin{equation*}
\sum_{h=1}^\infty \min(S_h, 1/h) = \infty.
\end{equation*}

\end{corollary}

For more background on the Duffin--Schaeffer conjecture and, more generally, on metric number theory, see Harman's monograph~\cite{harb}. A survey on recent results can be found in~\cite{bbdv}.\\

\section{Proof of Theorem}

Without loss of generality we will assume that 
$$
\frac{1}{n} \leq \psi(n) \leq \frac{1}{2} \quad \textrm{whenever $\psi(n) \neq 0$, for $n \geq 1$,}
$$
which is justified by the Erd\H os--Vaaler theorem~\cite{vaaler} and by~\cite[Theorem 2]{poll}. The following lemma is the Chung--Erd\H os inequality (see for example~\cite[Chapter 1.6]{zhe}).
\begin{lemma} \label{lemmace}
Let $A_1, \dots, A_N$ be events in a probability space. Then
$$
\p \left( \bigcup_{n=1}^N A_n \right) \geq \frac{ \left( \sum_{n=1}^N \p (A_n) \right)^2}{\sum_{n=1}^N \p(A_n) + 2 \sum_{1 \leq m < n \leq N} \p (A_m \cap A_n) }.
$$
\end{lemma}

We will use the Erd\H os--R\'enyi version of the Borel--Cantelli lemma in the following form (the main results of \cite{pet}, in the special case $H=0$).
\begin{lemma}\label{lemma2}
Let $A_1, A_2, \dots$ be events satisfying
$$
\sum_{n=1}^\infty \p(A_n) = \infty.
$$
Then
$$
\p \left(\limsup_{n \to \infty} A_n \right) \geq \limsup_{N \to \infty} \frac{\left( \sum_{n=1}^N \p(A_n) \right)^2}{2 \sum_{1 \leq m < n \leq N} \p (A_m \cap A_n)}.
$$
\end{lemma}


The following lemma is due to Strauch~\cite{strauch}, and has also been found independently by Pollington and Vaughan~\cite{poll}. We use the formulation from~\cite[Lemma 2]{bhhv}.
\begin{lemma} \label{lemma3}
For $m \neq n$ we have
$$
\lambda (\E_m \cap \E_n) \ll \lambda(\E_m) \lambda(\E_n) P(m,n),
$$
where
$$
P(m,n) = \prod_{\substack{p | mn/\gcd(m,n)^2, \\ p > D(m,n)}} \left(1 - \frac{1}{p}\right)^{-1} \qquad \textrm{with} \qquad D(m,n) = \frac{\max (n \psi(m),m \psi(n))}{\gcd(m,n)}.
$$
\end{lemma}

\begin{proof}[Proof of Theorem~\ref{th1}]
By~\eqref{dsc} there exists a number $j \in \{0,1,2\}$ such that
$$
\sum_{\substack{h \geq 1,\\h \equiv j \mod 3}} S_h = \infty.
$$
We will assume that $j=0$; the other cases can be treated in exactly the same way. We set
$$
B_h = \bigcup_{n=2^{2^{3h}}+1}^{2^{2^{3h+1}}} \E_n \qquad \textrm{and} \qquad T_h = \sum_{n=2^{2^{3h}}+1}^{2^{2^{3h+1}}} \frac{\psi(n) \varphi(n)}{n}, \qquad h \geq 1.
$$
Then $T_h = S_{3h},~h \geq 1,$ and $\sum_{h=1}^\infty T_h = \infty$. 
By Mertens's theorem \cite[Theorem 429]{hw} the function $P(m,n)$ in Lemma \ref{lemma3} is bounded by
$$
P(m,n) \ll \log \log (\max(m,n)).
$$
Thus there exists a constant $\hat{c}$ such that for $m,n \in \Delta_h$ for some $h$ and $m \neq n$ we have
$$
P(m,n) \leq \hat{c} h,
$$
and consequently 
\begin{equation} \label{imp}
\lambda(\E_m \cap \E_m) \leq \hat{c} h \lambda(\E_m) \lambda(\E_n).
\end{equation}
Thus by Lemma~\ref{lemmace}
\begin{eqnarray}
\lambda (B_h) & \geq & \frac{T_h^2}{T_h + 2 \sum_{\substack{m,n \in \Delta_{3h},\\m < n}} \lambda(\E_m \cap \E_m)} \nonumber\\
& \geq & \frac{T_h^2}{T_h + 6 \hat{c} h T_h^2}.\label{conc}
\end{eqnarray}
Note that~\eqref{condh} implies
$$
3 h T_h^2 \leq c T_h, \qquad h \geq 1.
$$
Consequently we can conclude from~\eqref{conc} that
\begin{equation} \label{she}
\lambda (B_h) \geq \frac{T_h^2}{T_h + 2 \hat{c} c T_h} \gg T_h,
\end{equation}
which in particular implies that
\begin{equation}  \label{bhi}
\sum_{h=1}^\infty \lambda(B_h) = \infty.
\end{equation} 
Now let $h_1 < h_2$. Then for $m \in \Delta_{3h_1}$ and $n \in \Delta_{3h_2}$ we have $n \geq m^4$. Thus
$$
D(m,n) \geq \frac{n \frac{1}{m}}{m} \geq \sqrt{n},
$$
which by Lemma~\ref{lemma3} implies
$$
\lambda(\E_m \cap \E_n) \ll \lambda(\E_m) \lambda (\E_n).
$$
Consequently we get
\begin{eqnarray}
& & \lambda (B_{h_1} \cap B_{h_2}) = \lambda \left( \left(\bigcup_{m \in \Delta_{3h_1}} \E_m \right) \cap \left(\bigcup_{n \in \Delta_{3h_2}} \E_n \right)  \right) \nonumber\\
& = & \lambda \left( \bigcup_{m \in \Delta_{3h_1},~n \in \Delta_{3h_2}} ( \E_m \cap \E_n) \right)  \nonumber\\
& \leq & \sum_{m \in \Delta_{3h_1}, ~n \in \Delta_{3h_2}} \lambda (\E_m \cap \E_n) \nonumber\\
& \ll & \sum_{m \in \Delta_{3h_1}, ~n \in \Delta_{3h_2}} \lambda (\E_m) \lambda(\E_n) \nonumber\\
& \ll & \underbrace{\left(\sum_{m \in \Delta_{3h_1}} \lambda(\E_m) \right)}_{=T_{h_1}} \underbrace{\left(\sum_{n \in \Delta_{3h_2}} \lambda(\E_n) \right)}_{=T_{h_2}} \nonumber\\
& \ll & \lambda(B_{h_1}) \lambda(B_{h_2}), \label{ll}
\end{eqnarray}
where for the last inequality we used~\eqref{she}. 
Note that clearly $\limsup_{h \to \infty} B_h \subset W(\psi)$. Thus, recalling \eqref{bhi} and applying Lemma~\ref{lemma2} to the sets $B_h, ~h \geq 1$, we obtain
\begin{eqnarray*}
\lambda(W(\psi)) \geq \underbrace{\limsup_{H \to \infty} \frac{ \left(\sum_{h=1}^H \lambda(B_h) \right)^2}{2 \sum_{1 \leq h_1 < h_2  \leq H} \lambda(B_{h_1} \cap B_{h_2})}}_{\gg 1 ~\textrm{by \eqref{ll}}} > 0. 
\end{eqnarray*}
By Gallagher's zero-one law~\cite{gall} the measure of the set $W(\psi)$ can only be either 0 or 1. Consequently we have $\lambda(W(\psi))=1$, which proves the theorem.
\end{proof}

\section*{Acknowledgements}
I want to thank Liangpan Li, who suggested Corollary 1.

\end{document}